\documentclass[11pt]{article}
\usepackage{graphicx}
\usepackage{csquotes}
\usepackage{indentfirst}
\usepackage{tikz}
\usetikzlibrary{braids}
\usepackage{tikz-cd}

\usepackage{amssymb,amsthm,amsmath}
\usepackage{xcolor,paralist,hyperref,fancyhdr,etoolbox}

\theoremstyle{plain}
\newtheorem{theorem}{Theorem}[]
\newtheorem{proposition}{Proposition}[]

\theoremstyle{definition}
\newtheorem{definition}{Definition}[subsection]
\newtheorem{example}{Example}[subsection]
\newtheorem{remark}{Remark}[subsection]

\hypersetup{ colorlinks=true, linkcolor=black, filecolor=black, urlcolor=black }

\begin{document}
\title{Derived Braids of Decreasing Products and their Categories} 
\author{Christopher Tapo}
\date{\vspace{-5ex}}

\maketitle


\begin{abstract}
Derived braids have been used to classify categorical structures based on the braid underlying a braided monoidal category \(\mathcal{V}\). With four-strand braids underlying the composition morphisms of tensor products of categories enriched over \(\mathcal{V}\), equality of derived braids has been seen to correspond with the results of Joyal and Street on coherence for braided monoidal categories, namely that diagrams commute when the braids underlying the legs of the diagram are equal. Equality of derived braids can then be seen as a generalization of the Yang-Baxter equation that appears in the work of Joyal and Street. The main result is a proof that decreasing products, which are braids formed from component braids with even numbers of strands, have equivalent derived braids. Plans for future work include interpreting what categorical structures correspond to decreasing products.
\end{abstract} 

\section{Introduction}
The work of Joyal and Street in \cite{JOYAL199320} shows a relationship between categorical structures and braids, particularly through braided monoidal categories and the Yang-Baxter equation. This relationship is studied further with generalizations of the Yang-Baxter equation as derived braids and their connection to categories enriched over braided monoidal categories as seen in the work of Stefan Forcey and Felita Humes in \cite{Forcey_2007}. A set of slides by Stefan Forcey \cite{DerivedBraids} focuses on derived braids and includes questions on what sorts of categorical structures might be possible when considering a braid \(x\) with any amount of strands that also satisfies the derived braid equality \(Lx=Rx\). Here we prove a conjecture appearing in the slides of Stefan Forcey on braids called "decreasing products" and show that they obey \(Lx=Rx\). We also state plans for future work on the categorical structures that can be interpreted from decreasing products.

We proceed as follows: First, we review derived braids and how they generalize the Yang-Baxter equation. Next, we define decreasing products and their component braids. Third, we prove that the component braids for decreasing products satisfy the derived braid equality \(Lx=Rx\). Next, we introduce the notions of combing decreasing products and embedding braids which are used to show that entire decreasing products also satisfy \(Lx=Rx\). Finally, we state plans for future work on the categorical interpretations for decreasing products and their derived braids.

\section{Derived Braids}
We shall consider braids with n strands in the braid group \(B_n\). The notion of k-ribbons is introduced in order to simplify the graphical presentation of braids with many strands. We then define left and right derived braids and show that the Yang-Baxter equation arises as the equality of left and right derived braids, or \(Lx=Rx\), for the single crossing of two strands. We end the section with some additional examples of derived braids, including ones that obey \(Lx=Rx\) and others that do not.

\subsection{k-ribbons}
\begin{definition} A \emph{k-ribbon} is a collection of \(k\) strands that are parallel and adjacent throughout a braid for \(k \in \mathbb{N}\). 
\end{definition}
A braid can be graphically reduced by representing k parallel strands with k-ribbons, which are represented by single strands with a the value \(k\) printed above the strand. This can be illustrated with the following example:
\\
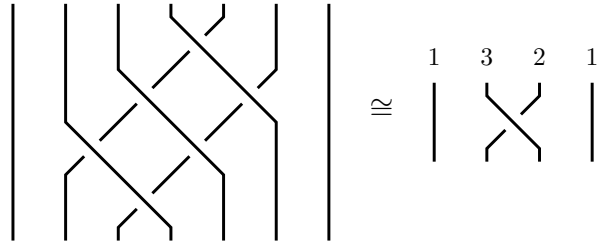
\begin{center}
\scalebox{0.70}{

\begin{tikzpicture}
\pic[
braid/.cd,
number of strands=7,
ultra thick,
gap=0.1,
control factor=0,
nudge factor=0,
name prefix=braid,
] at (2,0) {braid={ a_4, a_3-a_5, a_2-a_4, a_3}};

\node [scale=1.5] at (9,-2) {$\cong$};

\pic[
braid/number of strands=4,
ultra thick,
braid/gap=0.1,
braid/control factor=0,
braid/nudge factor=0,
name prefix=braid,
] at (10,-1.5) {braid={ a_2}};

\node[circle, draw=none, fill=white, inner sep=0pt, scale=1.25] at (10,-1) {\(1\)};
\node[circle, draw=none, fill=white, inner sep=0pt, scale=1.25] at (11,-1) {\(3\)};
\node[circle, draw=none, fill=white, inner sep=0pt, scale=1.25] at (12,-1) {\(2\)};
\node[circle, draw=none, fill=white, inner sep=0pt, scale=1.25] at (13,-1) {\(1\)};

\end{tikzpicture}

}

\end{center}
All braids from this point will be pictured in the reduced form if not already reduced, unless otherwise noted.

\subsection{Derived Braids and the Yang-Baxter equation}
We restrict our attention to the left and right derived braids as discussed in Forcey's slides \cite{DerivedBraids}, but other types of derived braids appear in \cite{Forcey_2007}. We then present various examples of derived braids and show whether they satisfy \(Lx=Rx\).

\begin{definition}
Given a braid \(x \in B_{2n}\) for \(n \in \mathbb{N}\), we define the \emph{left derived braid} \(Lx\) as the braid in \(B_{3n}\) formed from \(x\) by attaching n identity strands on the right side of \(x\), followed by pairing off strands at the bottom of \(x\) into \(n\) many 2-ribbons resulting in \(2n\) strands (considering 2-ribbons as strands), and then braiding the remaining \(2n\) strands according to the braid \(x\). 
\end{definition}
An example illustrates this process as follows:
\\
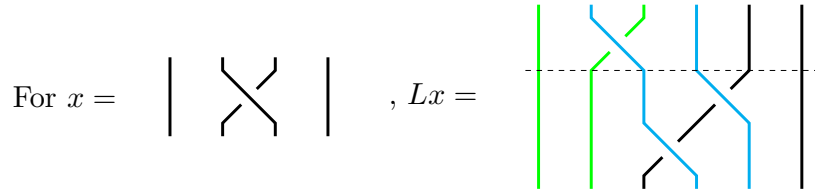
\begin{center}
\scalebox{0.70}{

\begin{tikzpicture}
\pic[
braid/.cd,
number of strands=4,
ultra thick,
gap=0.1,
control factor=0,
nudge factor=0,
name prefix=braid,
] at (2,-1) {braid={ a_2}};

\node [scale=1.5] at (0,-1.75) {For $x=$};
\node [scale=1.5] at (7,-1.75) {, $Lx=$};

\pic[
braid/number of strands=6,
ultra thick,
braid/strand 1/.style={green},
braid/strand 2/.style={cyan},
braid/strand 3/.style={green},
braid/strand 4/.style={cyan},
braid/gap=0.1,
braid/control factor=0,
braid/nudge factor=0,
name prefix=braid,
] at (9,0) {braid={ a_2, a_4, a_3}};

\path [draw=black, dashed] (8.75,-1.25) -- (14.25,-1.25);

\end{tikzpicture}

}

\end{center}
The dashed line separates the original braid \(x\) from the braiding described in the definition of left derived braid. Also, these braids are already in a graphically reduced form without k-ribbons, however strands for \(x\) are colored to help show the pairing of strands into 2-ribbons, which aren't 2-ribbons graphically as they do not remain parallel in \(x\). 

A \emph{right derived braid} \(Rx\) is defined similarly, except with n identity strands being adjoined on the left.We now include some examples of braids and show whether they obey the equality of left and right derived braids, \(Lx=Rx\).

\begin{example}
\hfill 
\begin{center}
\scalebox{0.70}{

\begin{tikzpicture}
\pic[
braid/.cd,
number of strands=6,
ultra thick,
gap=0.1,
control factor=0,
nudge factor=0,
name prefix=braid,
] at (1.5,0) {braid={ a_2-a_4}};

\node [scale=1.5] at (0,-0.75) {For $x=$};
\node [scale=1.5] at (10,-0.75) {, $Lx\neq Rx$ is shown as};

\end{tikzpicture}

}

\end{center}
\begin{center}
\scalebox{0.70}{

\begin{tikzpicture}
\pic[
braid/number of strands=9,
ultra thick,
braid/gap=0.1,
braid/control factor=0,
braid/nudge factor=0,
name prefix=braid,
] at (2,0) {braid={ a_2-a_4, a_4-a_7, a_3-a_5, a_4}};

\node [scale=1.5] at (11,-2) {$\neq$};

\pic[
braid/number of strands=9,
ultra thick,
braid/gap=0.1,
braid/control factor=0,
braid/nudge factor=0,
name prefix=braid,
] at (12,0) {braid={ a_5-a_7, a_2-a_5, a_4-a_6, a_5}};

\path [draw=black, dashed] (11.75,-1.25) -- (20.25,-1.25);
\path [draw=black, dashed] (1.75,-1.25) -- (10.25,-1.25);

\end{tikzpicture}

}

\end{center}
\begin{center}
\scalebox{0.70}{

\begin{tikzpicture}
\pic[
braid/.cd,
number of strands=4,
ultra thick,
gap=0.1,
control factor=0,
nudge factor=0,
name prefix=braid,
] at (1.5,0) {braid={ a_2, a_1-a_3, a_2, a_2, a_1-a_3}};

\node [scale=1.5] at (0,-2.75) {For $x=$};
\node [scale=1.5] at (8,-2.75) {, $Lx=Rx$ is shown as};

\end{tikzpicture}

}

\end{center}
\begin{center}
\scalebox{0.70}{

\begin{tikzpicture}
\pic[
braid/number of strands=6,
ultra thick,
braid/gap=0.1,
braid/control factor=0,
braid/nudge factor=0,
name prefix=braid,
] at (2,0) {braid={ a_2, a_1-a_3, a_2, a_2, a_1-a_3, a_4, a_3-a_5, a_2-a_4, a_1-a_3, a_2, a_2, a_3, a_1-a_4, a_2-a_5}};

\node [scale=1.5] at (9.5,-7) {$=$};

\pic[
braid/number of strands=6,
ultra thick,
braid/gap=0.1,
braid/control factor=0,
braid/nudge factor=0,
name prefix=braid,
] at (12,0) {braid={ a_4, a_3-a_5, a_4, a_4, a_3-a_5, a_2, a_1-a_3, a_2-a_4, a_3-a_5, a_4, a_4, a_3, a_2-a_5, a_1-a_4}};

\path [draw=black, dashed] (11.75,-5.25) -- (17.25,-5.25);
\path [draw=black, dashed] (1.75,-5.25) -- (7.25,-5.25);

\end{tikzpicture}

}

\end{center}
\begin{center}
\scalebox{0.70}{

\begin{tikzpicture}
\pic[
braid/.cd,
number of strands=2,
ultra thick,
gap=0.1,
control factor=0,
nudge factor=0,
name prefix=braid,
] at (1.5,-2) {braid={ a_1}};

\node [scale=1.5] at (0,-2.75) {For $x=$};
\node [scale=1.5] at (6,-2.75) {, $Lx=Rx$ is shown as};

\end{tikzpicture}

}

\end{center}
\begin{center}
\scalebox{0.70}{

\begin{tikzpicture}
\pic[
braid/number of strands=3,
ultra thick,
braid/gap=0.1,
braid/control factor=0,
braid/nudge factor=0,
name prefix=braid,
] at (2,0) {braid={ a_1, a_2, a_1}};

\node [scale=1.5] at (5,-2) {$=$};

\pic[
braid/number of strands=3,
ultra thick,
braid/gap=0.1,
braid/control factor=0,
braid/nudge factor=0,
name prefix=braid,
] at (6,0) {braid={ a_2, a_1, a_2}};

\path [draw=black, dashed] (5.75,-1.25) -- (8.25,-1.25);
\path [draw=black, dashed] (1.75,-1.25) -- (4.25,-1.25);

\end{tikzpicture}

}

\end{center}
\end{example}
The last example is the \emph{Yang-Baxter equation} which can now be seen as the equality of left and right derived braids of the two-stranded braid with a single crossing.

\section{Decreasing Products}
Here we define decreasing products and their component braids, followed by a proof that decreasing product components satisfy \(Lx=Rx\). We then introduce the method of combing decreasing products and the notion of embedded braids which will be used to prove our main result.

\subsection{Decreasing Product Components}
In order to define decreasing products, we need to first introduce their component braids. \emph{Decreasing product components} are the braids \(b_0, b_1,...,b_k \in B_{2n}\) for \(k=\frac{n}{2}\) if \(n\) is even and \(k=\frac{n-1}{2}\) if \(n\) is odd, shown as follows:
\\
\begin{center}
\scalebox{0.70}{

\begin{tikzpicture}
\pic[
braid/.cd,
number of strands=2,
ultra thick,
gap=0.1,
control factor=0,
nudge factor=0,
name prefix=braid,
] at (1.5,-2) {braid={ a_1}};

\pic[
braid/.cd,
number of strands=4,
ultra thick,
gap=0.1,
control factor=0,
nudge factor=0,
name prefix=braid,
] at (5.5,-2) {braid={ a_2}};

\pic[
braid/.cd,
number of strands=4,
ultra thick,
gap=0.1,
control factor=0,
nudge factor=0,
name prefix=braid,
] at (11.5,-2) {braid={ a_2}};

\pic[
braid/.cd,
number of strands=4,
ultra thick,
gap=0.1,
control factor=0,
nudge factor=0,
name prefix=braid,
] at (2.5,-5) {braid={ a_2}};

\pic[
braid/.cd,
number of strands=4,
ultra thick,
gap=0.1,
control factor=0,
nudge factor=0,
name prefix=braid,
] at (10.5,-5) {braid={ a_2}};

\node [scale=1.5] at (0,-2.75) {$b_0=$};
\node [scale=1.5] at (4,-2.75) {, $b_1=$};
\node [scale=1.5] at (10,-2.75) {, $b_2=$};
\node [scale=1.5] at (0,-5.75) { $b_{\frac{n}{2}}$ ($n$ even) =};
\node [scale=1.5] at (8,-5.75) {, $b_{\frac{n-1}{2}}$ ($n$ odd) =};

\node[circle, draw=none, fill=white, inner sep=0pt, scale=1.25] at (1.5,-1.5) {\(n\)};
\node[circle, draw=none, fill=white, inner sep=0pt, scale=1.25] at (2.5,-1.5) {\(n\)};

\node[circle, draw=none, fill=white, inner sep=0pt, scale=1.25] at (5.5,-1.5) {\(1\)};
\node[circle, draw=none, fill=white, inner sep=0pt, scale=1] at (6.5,-1.5) {\(n-1\)};
\node[circle, draw=none, fill=white, inner sep=0pt, scale=1] at (7.5,-1.5) {\(n-1\)};
\node[circle, draw=none, fill=white, inner sep=0pt, scale=1.25] at (8.5,-1.5) {\(1\)};

\node[circle, draw=none, fill=white, inner sep=0pt, scale=1.25] at (11.5,-1.5) {\(2\)};
\node[circle, draw=none, fill=white, inner sep=0pt, scale=1] at (12.5,-1.5) {\(n-2\)};
\node[circle, draw=none, fill=white, inner sep=0pt, scale=1] at (13.5,-1.5) {\(n-2\)};
\node[circle, draw=none, fill=white, inner sep=0pt, scale=1.25] at (14.5,-1.5) {\(2\)};

\node[circle, draw=none, fill=white, inner sep=0pt, scale=1.25] at (2.5,-4.5) {\(\frac{n}{2}\)};
\node[circle, draw=none, fill=white, inner sep=0pt, scale=1.25] at (3.5,-4.5) {\(\frac{n}{2}\)};
\node[circle, draw=none, fill=white, inner sep=0pt, scale=1.25] at (4.5,-4.5) {\(\frac{n}{2}\)};
\node[circle, draw=none, fill=white, inner sep=0pt, scale=1.25] at (5.5,-4.5) {\(\frac{n}{2}\)};

\node[circle, draw=none, fill=white, inner sep=0pt, scale=1.25] at (10.5,-4.5) {\(\frac{n-1}{2}\)};
\node[circle, draw=none, fill=white, inner sep=0pt, scale=1.25] at (11.5,-4.5) {\(\frac{n+1}{2}\)};
\node[circle, draw=none, fill=white, inner sep=0pt, scale=1.25] at (12.5,-4.5) {\(\frac{n+1}{2}\)};
\node[circle, draw=none, fill=white, inner sep=0pt, scale=1.25] at (13.5,-4.5) {\(\frac{n-1}{2}\)};

\end{tikzpicture}

}

\end{center}
\begin{definition}
\emph{Decreasing products} are defined as the combinations of decreasing product components \(b_{i_1},b_{i_2},...b_{i_k}\) where \(i_j < i_l\) for \(0\le j<l\le k\).
\end{definition}

To prove that decreasing products obey \(Lx=Rx\), we start by showing that decreasing product components obey \(Lx=Rx\).

\begin{proposition}
Given a decreasing product component \(b_k \in B_{2n}\), the equality of derived braids \(L(b_k)=R(b_k)\) is satisfied.
\end{proposition}

\begin{proof}
\\
\\
\emph{Case 1. k=0}
\\
\\
\indent \(L(b_k)=R(b_k)\) follows from the Yang-Baxter equation on n-ribbons:
\\
\begin{center}
\scalebox{0.70}{

\begin{tikzpicture}
\pic[
braid/number of strands=3,
ultra thick,
braid/gap=0.1,
braid/control factor=0,
braid/nudge factor=0,
name prefix=braid,
] at (2,0) {braid={ a_1, a_2, a_1}};

\node [scale=1.5] at (5,-2) {$=$};

\pic[
braid/number of strands=3,
ultra thick,
braid/gap=0.1,
braid/control factor=0,
braid/nudge factor=0,
name prefix=braid,
] at (6,0) {braid={ a_2, a_1, a_2}};

\path [draw=black, dashed] (5.75,-1.25) -- (8.25,-1.25);
\path [draw=black, dashed] (1.75,-1.25) -- (4.25,-1.25);

\node[circle, draw=none, fill=white, inner sep=0pt, scale=1.25] at (2,0.5) {\(n\)};
\node[circle, draw=none, fill=white, inner sep=0pt, scale=1.25] at (3,0.5) {\(n\)};
\node[circle, draw=none, fill=white, inner sep=0pt, scale=1.25] at (4,0.5) {\(n\)};

\node[circle, draw=none, fill=white, inner sep=0pt, scale=1.25] at (6,0.5) {\(n\)};
\node[circle, draw=none, fill=white, inner sep=0pt, scale=1.25] at (7,0.5) {\(n\)};
\node[circle, draw=none, fill=white, inner sep=0pt, scale=1.25] at (8,0.5) {\(n\)};

\end{tikzpicture}

}

\end{center}
\emph{Case 2. $0<k\le\frac{n-1}{2}$ if n is odd or $0<k<\frac{n}{2}$ if n is even}
\\
\\
\indent Here the braid \(b_k\) is of the form:
\begin{center}
\scalebox{0.70}{

\begin{tikzpicture}
\pic[
braid/number of strands=4,
ultra thick,
braid/gap=0.1,
braid/control factor=0,
braid/nudge factor=0,
name prefix=braid,
] at (2,0) {braid={ a_2}};

\node[circle, draw=none, fill=white, inner sep=0pt, scale=1.25] at (2,0.5) {\(p\)};
\node[circle, draw=none, fill=white, inner sep=0pt, scale=1.25] at (3,0.5) {\(q\)};
\node[circle, draw=none, fill=white, inner sep=0pt, scale=1.25] at (4,0.5) {\(q\)};
\node[circle, draw=none, fill=white, inner sep=0pt, scale=1.25] at (5,0.5) {\(p\)};

\end{tikzpicture}

}

\end{center}
with \(p<q\).
\\
\\
\indent \(L(b_k)=R(b_k)\) follows from the following braid equality:
\\
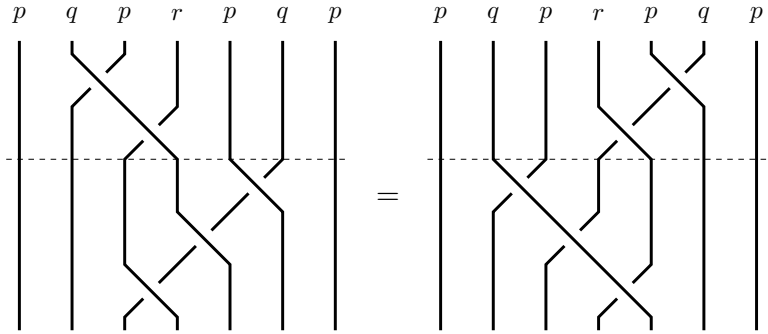
\begin{center}
\scalebox{0.70}{

\begin{tikzpicture}
\pic[
braid/number of strands=7,
ultra thick,
braid/gap=0.1,
braid/control factor=0,
braid/nudge factor=0,
name prefix=braid,
] at (2,0) {braid={ a_2, a_3, a_5, a_4, a_3}};

\node [scale=1.5] at (9,-3) {$=$};

\pic[
braid/number of strands=7,
ultra thick,
braid/gap=0.1,
braid/control factor=0,
braid/nudge factor=0,
name prefix=braid,
] at (10,0) {braid={ a_5, a_4, a_2, a_3, a_4}};

\path [draw=black, dashed] (1.75,-2.25) -- (8.25,-2.25);
\path [draw=black, dashed] (9.75,-2.25) -- (16.25,-2.25);

\node[circle, draw=none, fill=white, inner sep=0pt, scale=1.25] at (2,0.5) {\(p\)};
\node[circle, draw=none, fill=white, inner sep=0pt, scale=1.25] at (3,0.5) {\(q\)};
\node[circle, draw=none, fill=white, inner sep=0pt, scale=1.25] at (4,0.5) {\(p\)};
\node[circle, draw=none, fill=white, inner sep=0pt, scale=1.25] at (5,0.5) {\(r\)};
\node[circle, draw=none, fill=white, inner sep=0pt, scale=1.25] at (6,0.5) {\(p\)};
\node[circle, draw=none, fill=white, inner sep=0pt, scale=1.25] at (7,0.5) {\(q\)};
\node[circle, draw=none, fill=white, inner sep=0pt, scale=1.25] at (8,0.5) {\(p\)};

\node[circle, draw=none, fill=white, inner sep=0pt, scale=1.25] at (10,0.5) {\(p\)};
\node[circle, draw=none, fill=white, inner sep=0pt, scale=1.25] at (11,0.5) {\(q\)};
\node[circle, draw=none, fill=white, inner sep=0pt, scale=1.25] at (12,0.5) {\(p\)};
\node[circle, draw=none, fill=white, inner sep=0pt, scale=1.25] at (13,0.5) {\(r\)};
\node[circle, draw=none, fill=white, inner sep=0pt, scale=1.25] at (14,0.5) {\(p\)};
\node[circle, draw=none, fill=white, inner sep=0pt, scale=1.25] at (15,0.5) {\(q\)};
\node[circle, draw=none, fill=white, inner sep=0pt, scale=1.25] at (16,0.5) {\(p\)};

\end{tikzpicture}

}

\end{center}
Where \(q-p=r\).
\\
\\
\emph{Case 3. k=$\frac{n}{2}$}
\\
\\
\indent \(L(b_k)=R(b_k)\) follows from the following braid equality:
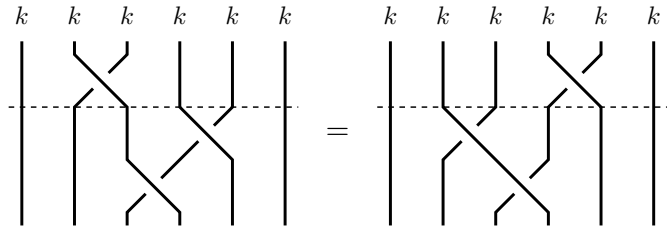
\begin{center}
\scalebox{0.70}{

\begin{tikzpicture}
\pic[
braid/number of strands=6,
ultra thick,
braid/gap=0.1,
braid/control factor=0,
braid/nudge factor=0,
name prefix=braid,
] at (2,0) {braid={ a_2, a_4, a_3}};

\node [scale=1.5] at (8,-1.75) {$=$};

\pic[
braid/number of strands=6,
ultra thick,
braid/gap=0.1,
braid/control factor=0,
braid/nudge factor=0,
name prefix=braid,
] at (9,0) {braid={ a_4, a_2, a_3}};

\path [draw=black, dashed, thick] (1.75,-1.25) -- (7.25,-1.25);
\path [draw=black, dashed, thick] (8.75,-1.25) -- (14.25,-1.25);

\node[circle, draw=none, fill=white, inner sep=0pt, scale=1.25] at (2,0.5) {\(k\)};
\node[circle, draw=none, fill=white, inner sep=0pt, scale=1.25] at (3,0.5) {\(k\)};
\node[circle, draw=none, fill=white, inner sep=0pt, scale=1.25] at (4,0.5) {\(k\)};
\node[circle, draw=none, fill=white, inner sep=0pt, scale=1.25] at (5,0.5) {\(k\)};
\node[circle, draw=none, fill=white, inner sep=0pt, scale=1.25] at (6,0.5) {\(k\)};
\node[circle, draw=none, fill=white, inner sep=0pt, scale=1.25] at (7,0.5) {\(k\)};

\node[circle, draw=none, fill=white, inner sep=0pt, scale=1.25] at (9,0.5) {\(k\)};
\node[circle, draw=none, fill=white, inner sep=0pt, scale=1.25] at (10,0.5) {\(k\)};
\node[circle, draw=none, fill=white, inner sep=0pt, scale=1.25] at (11,0.5) {\(k\)};
\node[circle, draw=none, fill=white, inner sep=0pt, scale=1.25] at (12,0.5) {\(k\)};
\node[circle, draw=none, fill=white, inner sep=0pt, scale=1.25] at (13,0.5) {\(k\)};
\node[circle, draw=none, fill=white, inner sep=0pt, scale=1.25] at (14,0.5) {\(k\)};

\end{tikzpicture}

}

\end{center}

\end{proof}

Before proving our main result, we will introduce the notions of combing and embedding braids. We now include some examples of decreasing products, each of which obeys \(Lx=Rx\). 

\begin{example}
\hfill
\begin{center}
\scalebox{0.70}{

\begin{tikzpicture}
\pic[
braid/.cd,
number of strands=4,
ultra thick,
gap=0.1,
control factor=0,
nudge factor=0,
name prefix=braid,
] at (2.5,-1) {braid={ a_2}};

\node [scale=1.5] at (0,-1.75) {For $b_2$ ($n=5$) =};
\node [scale=1.5] at (9,-1.75) {, $Lx=Rx$ is shown as};

\node[circle, draw=none, fill=white, inner sep=0pt, scale=1.25] at (2.5,-0.5) {\(2\)};
\node[circle, draw=none, fill=white, inner sep=0pt, scale=1.25] at (3.5,-0.5) {\(3\)};
\node[circle, draw=none, fill=white, inner sep=0pt, scale=1.25] at (4.5,-0.5) {\(3\)};
\node[circle, draw=none, fill=white, inner sep=0pt, scale=1.25] at (5.5,-0.5) {\(2\)};

\end{tikzpicture}

}

\end{center}
\begin{center}
\scalebox{0.70}{

\begin{tikzpicture}
\pic[
braid/number of strands=7,
ultra thick,
braid/gap=0.1,
braid/control factor=0,
braid/nudge factor=0,
name prefix=braid,
] at (2,0) {braid={ a_2, a_3, a_5, a_4, a_3}};

\node [scale=1.5] at (9,-3) {$=$};

\pic[
braid/number of strands=7,
ultra thick,
braid/gap=0.1,
braid/control factor=0,
braid/nudge factor=0,
name prefix=braid,
] at (10,0) {braid={ a_5, a_4, a_2, a_3, a_4}};

\path [draw=black, dashed] (9.75,-2.25) -- (16.25,-2.25);
\path [draw=black, dashed] (1.75,-2.25) -- (8.25,-2.25);

\node[circle, draw=none, fill=white, inner sep=0pt, scale=1.25] at (2,0.5) {\(2\)};
\node[circle, draw=none, fill=white, inner sep=0pt, scale=1.25] at (3,0.5) {\(3\)};
\node[circle, draw=none, fill=white, inner sep=0pt, scale=1.25] at (4,0.5) {\(2\)};
\node[circle, draw=none, fill=white, inner sep=0pt, scale=1.25] at (5,0.5) {\(1\)};
\node[circle, draw=none, fill=white, inner sep=0pt, scale=1.25] at (6,0.5) {\(2\)};
\node[circle, draw=none, fill=white, inner sep=0pt, scale=1.25] at (7,0.5) {\(3\)};
\node[circle, draw=none, fill=white, inner sep=0pt, scale=1.25] at (8,0.5) {\(2\)};

\node[circle, draw=none, fill=white, inner sep=0pt, scale=1.25] at (10,0.5) {\(2\)};
\node[circle, draw=none, fill=white, inner sep=0pt, scale=1.25] at (11,0.5) {\(3\)};
\node[circle, draw=none, fill=white, inner sep=0pt, scale=1.25] at (12,0.5) {\(2\)};
\node[circle, draw=none, fill=white, inner sep=0pt, scale=1.25] at (13,0.5) {\(1\)};
\node[circle, draw=none, fill=white, inner sep=0pt, scale=1.25] at (14,0.5) {\(2\)};
\node[circle, draw=none, fill=white, inner sep=0pt, scale=1.25] at (15,0.5) {\(3\)};
\node[circle, draw=none, fill=white, inner sep=0pt, scale=1.25] at (16,0.5) {\(2\)};

\end{tikzpicture}

}

\end{center}
\begin{center}
\scalebox{0.70}{

\begin{tikzpicture}
\pic[
braid/.cd,
number of strands=6,
ultra thick,
gap=0.1,
control factor=0,
nudge factor=0,
name prefix=braid,
] at (3.5,-1) {braid={ a_3, a_2-a_4, a_3, a_3}};

\node [scale=1.5] at (0,-1.75) {For $b_1$$b_2$ ($n=5$) =};
\node [scale=1.5] at (12,-1.75) {, $Lx=Rx$ is shown as};

\node[circle, draw=none, fill=white, inner sep=0pt, scale=1.25] at (3.5,-0.5) {\(1\)};
\node[circle, draw=none, fill=white, inner sep=0pt, scale=1.25] at (4.5,-0.5) {\(3\)};
\node[circle, draw=none, fill=white, inner sep=0pt, scale=1.25] at (5.5,-0.5) {\(1\)};
\node[circle, draw=none, fill=white, inner sep=0pt, scale=1.25] at (6.5,-0.5) {\(1\)};
\node[circle, draw=none, fill=white, inner sep=0pt, scale=1.25] at (7.5,-0.5) {\(3\)};
\node[circle, draw=none, fill=white, inner sep=0pt, scale=1.25] at (8.5,-0.5) {\(1\)};

\end{tikzpicture}

}

\end{center}
\begin{center}
\scalebox{0.70}{

\begin{tikzpicture}
\pic[
braid/number of strands=9,
ultra thick,
braid/gap=0.1,
braid/control factor=0,
braid/nudge factor=0,
name prefix=braid,
] at (2,0) {braid={ a_3, a_2-a_4, a_3, a_3, a_6, a_5-a_7, a_4-a_6, a_3-a_5, a_4, a_4, a_5}};

\node [scale=1.5] at (11,-6) {$=$};

\pic[
braid/number of strands=9,
ultra thick,
braid/gap=0.1,
braid/control factor=0,
braid/nudge factor=0,
name prefix=braid,
] at (12,0) {braid={ a_6, a_5-a_7, a_6, a_6, a_3, a_2-a_4, a_3-a_5, a_4-a_6, a_5, a_5, a_4}};

\path [draw=black, dashed] (11.75,-4.25) -- (20.25,-4.25);
\path [draw=black, dashed] (1.75,-4.25) -- (10.25,-4.25);

\node[circle, draw=none, fill=white, inner sep=0pt, scale=1.25] at (2,0.5) {\(1\)};
\node[circle, draw=none, fill=white, inner sep=0pt, scale=1.25] at (3,0.5) {\(3\)};
\node[circle, draw=none, fill=white, inner sep=0pt, scale=1.25] at (4,0.5) {\(1\)};
\node[circle, draw=none, fill=white, inner sep=0pt, scale=1.25] at (5,0.5) {\(1\)};
\node[circle, draw=none, fill=white, inner sep=0pt, scale=1.25] at (6,0.5) {\(3\)};
\node[circle, draw=none, fill=white, inner sep=0pt, scale=1.25] at (7,0.5) {\(1\)};
\node[circle, draw=none, fill=white, inner sep=0pt, scale=1.25] at (8,0.5) {\(1\)};
\node[circle, draw=none, fill=white, inner sep=0pt, scale=1.25] at (9,0.5) {\(3\)};
\node[circle, draw=none, fill=white, inner sep=0pt, scale=1.25] at (10,0.5) {\(1\)};

\node[circle, draw=none, fill=white, inner sep=0pt, scale=1.25] at (12,0.5) {\(1\)};
\node[circle, draw=none, fill=white, inner sep=0pt, scale=1.25] at (13,0.5) {\(3\)};
\node[circle, draw=none, fill=white, inner sep=0pt, scale=1.25] at (14,0.5) {\(1\)};
\node[circle, draw=none, fill=white, inner sep=0pt, scale=1.25] at (15,0.5) {\(1\)};
\node[circle, draw=none, fill=white, inner sep=0pt, scale=1.25] at (16,0.5) {\(3\)};
\node[circle, draw=none, fill=white, inner sep=0pt, scale=1.25] at (17,0.5) {\(1\)};
\node[circle, draw=none, fill=white, inner sep=0pt, scale=1.25] at (18,0.5) {\(1\)};
\node[circle, draw=none, fill=white, inner sep=0pt, scale=1.25] at (19,0.5) {\(3\)};
\node[circle, draw=none, fill=white, inner sep=0pt, scale=1.25] at (20,0.5) {\(1\)};

\end{tikzpicture}

}

\end{center}

\end{example}

\subsection{Combing Decreasing Products}
Combing braids is a form of braid isotopy yielding equivalent braids in the braid group, as seen in the work of Artin \cite{10362252-eeea-365a-ac20-0098364f4689}. To show that the left and derived braids for a decreasing product are equivalent, we use a method of combing the the derived braids into components that satisfy the equalities for decreasing product components seen in Proposition 1.

Consider a left derived braid \(Lx\) of a decreasing product \(x=b_{i_1}b_{i_2}...b_{i_k}\) as having sections corresponding to the decreasing product components \(b_{i_j}\) for  \(j \in 1,...,k\) as well as sections corresponding to the braiding from the derived braid based on those components, denoted \(b_{i_j}^*\). We can comb each \(b_{i_j}\) of \(Lx\) through sections so that it becomes an equivalent braid \(Lx^{comb}=b_{i_1}b_{i_1}^*b_{i_2}^{comb}b_{i_2}^*...b_{i_k}^{comb}b_{i_k}^*\) where each \(b_{i_j}\) with \(j>1\) is combed through sections transforming the braid component into a combed version denoted \(b_{i_j}^{comb}\). Combing for a right derived braid \(Rx\) is analogous. The process of combing is illustrated below, where the following braids are not shown in the reduced form using k-ribbons.

\begin{example}
\hfill

\begin{center}
\scalebox{0.60}{

\begin{tikzpicture}
\pic[
braid/number of strands=9,
ultra thick,
braid/gap=0.1,
braid/control factor=0,
braid/nudge factor=0,
name prefix=braid,
] at (2,0) {braid={ a_3, a_2-a_4, a_1-a_3-a_5, a_2-a_4, a_3, a_3, a_2-a_4, a_3, a_6, a_5-a_7, a_4-a_6-a_8, a_3-a_5-a_7, a_2-a_4-a_6, a_1-a_3-a_5, a_2-a_4, a_3, a_3, a_2-a_4, a_3-a_5, a_4-a_6, a_5}};

\node [scale=2] at (11.5,-11) {$\stackrel{\text{comb}}{\Rightarrow}$};

\pic[
braid/number of strands=9,
ultra thick,
braid/gap=0.1,
braid/control factor=0,
braid/nudge factor=0,
name prefix=braid,
] at (13,0) {braid={ a_3, a_2-a_4, a_1-a_3-a_5, a_2-a_4, a_3, a_6, a_5-a_7, a_4-a_6-a_8, a_3-a_5-a_7, a_2-a_4-a_6, a_1-a_3-a_5, a_2-a_4, a_3, a_6, a_5-a_7, a_6, a_3, a_2-a_4, a_3-a_5, a_4-a_6, a_5}};

\path [draw=black, dashed] (12.75,-5.25) -- (23,-5.25);
\path [draw=black, dashed] (0,-5.25) -- (10.25,-5.25);

\path [draw=black, dashed] (12.75,-13.25) -- (23,-13.25);
\path [draw=black, dashed] (0,-8.25) -- (10.25,-8.25);

\path [draw=black, dashed] (12.75,-16.25) -- (23,-16.25);
\path [draw=black, dashed] (0,-16.25) -- (10.25,-16.25);

\node[circle, draw=none, fill=white, inner sep=0pt, scale=2] at (0.5,-2.5) {\(b_0\)};
\node[circle, draw=none, fill=white, inner sep=0pt, scale=2] at (22.25,-2.5) {\(b_0\)};
\node[circle, draw=none, fill=white, inner sep=0pt, scale=2] at (0.5,-6.75) {\(b_1\)};
\node[circle, draw=none, fill=white, inner sep=0pt, scale=2] at (22.25,-8.5) {\(b_0^*\)};
\node[circle, draw=none, fill=white, inner sep=0pt, scale=2] at (0.5,-12) {\(b_0^*\)};
\node[circle, draw=none, fill=white, inner sep=0pt, scale=2] at (22.25,-14.75) {\(b_1^{comb}\)};
\node[circle, draw=none, fill=white, inner sep=0pt, scale=2] at (0.5,-19) {\(b_1^*\)};
\node[circle, draw=none, fill=white, inner sep=0pt, scale=2] at (22.25,-19) {\(b_1^*\)};

\end{tikzpicture}

}

\end{center}

\end{example}
It is important to note that the combination \(b_1^{comb}b_1^*\) in \(Lx^{comb}\) as pictured above is equal to the right derived braid \(R(b_1)\) over \(b_1 \in B_9\). This is not true in general for combinations of combed components, however, these combinations are equivalent to a left or right derived braid of a decreasing product component with either less or equal strands embedded into \(3n\) strands. This can be checked using the cases in the proof of Proposition 1, while the notion of embedding braids is detailed in the following section.

\subsection{Embedded Braids}
We shall consider braids in \(B_m\) \emph{embedded} in \(B_n\) for \(m<n\) by attaching \(n-m\) many identity strands. Let \(b \in B_m\) embedded in \(B_n\) be denoted \(b^m \in B_n\). More decorations could specify the exact manner in which the identity strands are attached, but this is not needed for proving our main result and thus omitted to simplify notation. 

\begin{remark}
The braid appearing in the second case of Proposition 1 and in the first example of Example 3.1.1. can be rewritten using braid isotopy to reveal a Yang-Baxter equation as an embedded braid:

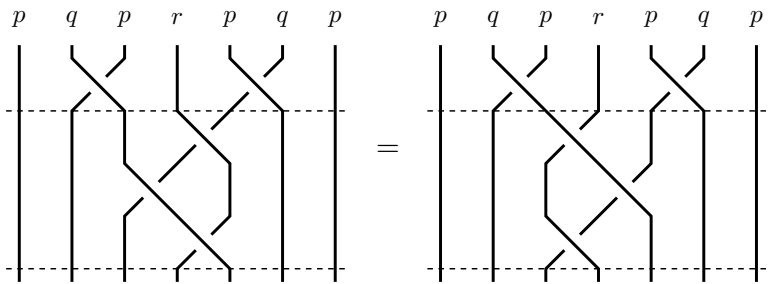
\begin{center}
\scalebox{0.70}{

\begin{tikzpicture}
\pic[
braid/.cd,
number of strands=7,
ultra thick,
gap=0.1,
control factor=0,
nudge factor=0,
name prefix=braid,
] at (2,0) {braid={ a_2-a_5 | a_4 | a_3 | a_4}};

\node [scale=1.5] at (9,-2) {$=$};

\pic[
braid/number of strands=7,
ultra thick,
braid/gap=0.1,
braid/control factor=0,
braid/nudge factor=0,
name prefix=braid,
] at (10,0) {braid={ a_2-a_5 | a_3 | a_4 | a_3 }};

\path [draw=black, dashed, thick] (9.75,-4.25) -- (16.25,-4.25);
\path [draw=black, dashed, thick] (9.75,-1.25) -- (16.25,-1.25);
\path [draw=black, dashed, thick] (1.75,-4.25) -- (8.25,-4.25);
\path [draw=black, dashed, thick] (1.75,-1.25) -- (8.25,-1.25);

\node[circle, draw=none, fill=white, inner sep=0pt, scale=1.25] at (2,0.5) {\(p\)};
\node[circle, draw=none, fill=white, inner sep=0pt, scale=1.25] at (3,0.5) {\(q\)};
\node[circle, draw=none, fill=white, inner sep=0pt, scale=1.25] at (4,0.5) {\(p\)};
\node[circle, draw=none, fill=white, inner sep=0pt, scale=1.25] at (5,0.5) {\(r\)};
\node[circle, draw=none, fill=white, inner sep=0pt, scale=1.25] at (6,0.5) {\(p\)};
\node[circle, draw=none, fill=white, inner sep=0pt, scale=1.25] at (7,0.5) {\(q\)};
\node[circle, draw=none, fill=white, inner sep=0pt, scale=1.25] at (8,0.5) {\(p\)};

\node[circle, draw=none, fill=white, inner sep=0pt, scale=1.25] at (10,0.5) {\(p\)};
\node[circle, draw=none, fill=white, inner sep=0pt, scale=1.25] at (11,0.5) {\(q\)};
\node[circle, draw=none, fill=white, inner sep=0pt, scale=1.25] at (12,0.5) {\(p\)};
\node[circle, draw=none, fill=white, inner sep=0pt, scale=1.25] at (13,0.5) {\(r\)};
\node[circle, draw=none, fill=white, inner sep=0pt, scale=1.25] at (14,0.5) {\(p\)};
\node[circle, draw=none, fill=white, inner sep=0pt, scale=1.25] at (15,0.5) {\(q\)};
\node[circle, draw=none, fill=white, inner sep=0pt, scale=1.25] at (16,0.5) {\(p\)};

\end{tikzpicture}

}

\end{center}

\end{remark}

\subsection{Theorem for Decreasing Products}

We are now ready to prove the main theorem.

\begin{theorem}
Given a decreasing product \(x=b_{i_1}b_{i_2}...b_{i_k} \in B_{2n}\), the equality of derived braids \(Lx=Rx\) is satisfied.
\end{theorem}

\begin{proof}
The derived braids of \(x\) are equivalent to an alternating combination of derived braids of decreasing product components embedded in \(B_{3n}\) through combing as follows:
\begin{center}

\(L(b_{i_1}b_{i_2}...b_{i_k})=L(b_{i_1})R(b_{i_2}^{comb})^{m_2}...L(b_{i_k}^{comb})^{m_k}\) for odd \(k\), or \(R(b_{i_k}^{comb})^{m_k}\) for even \(k\)  

\end{center}
\indent and
\begin{center}

\(R(b_{i_1}b_{i_2}...b_{i_k})=R(b_{i_1})L(b_{i_2}^{comb})^{m_2}...R(b_{i_k}^{comb})^{m_k}\) for odd \(k\), or \(L(b_{i_k}^{comb})^{m_k}\) for even \(k\) 

\end{center}
for \(i_j^{'} \le i_j\) and \(n \ge m_j \ge m_l\) for \(j<l\). Note that \(b_{i_1}\) stays at the top of the braid and is never combed, hence \(b_{i_1}\) corresponds with \(L(b_{i_1})\) and \(R(b_{i_1})\). Since \(L(b_{i_1})=R(b_{i_1})\) and \(L(b_{i_j}^{comb})^{m_j}=R(b_{i_j}^{comb})^{m_j}\) for all \(j\) due to Proposition 1, it follows that \(Lx=Rx\) for all \(x=b_{i_1}b_{i_2}...b_{i_k} \in B_{2n}\).

\end{proof}

\section{Future Work}
With the main result showing that decreasing products always obey \(Lx=Rx\), this is a starting point for future work on finding possible categorical structures related to decreasing products and their derived braids. Additional plans for future work include investigations of similar categorical structures, like tortile tensor categories described in \cite{SHUM199457}, to find other potential connections to derived braids. Regarding the relationship between derived braids and categorical enrichment, other braid constructions are also planned to be investigated alongside decreasing products in order to further generalize ideas on tensor products of categories enriched over a braided monoidal category \(\mathcal{V}\) as seen in \cite{Forcey_2007}. While considering derived braids as a generalization of the Yang-Baxter equation, it may be worthwhile to investigate possible applications to braided Frobenius algebras and topological quantum field theories, as well as various topics in knot theory.

\bibliographystyle{plain}
\bibliography{bibliography.bib}
\nocite{*}

\end{document}